\documentclass{amsart}

\usepackage{amsfonts}
\usepackage{amsmath}
\usepackage{amssymb}
\usepackage{amsthm}
\usepackage{latexsym}

\newtheorem{Thm}{Theorem}
\newtheorem{Lem}[Thm]{Lemma}
\newtheorem{Prop}[Thm]{Proposition}
\newtheorem{Cor}[Thm]{Corollary}

\theoremstyle{definition}

\theoremstyle{remark}
\newtheorem{Rem}[Thm]{Remark}

\numberwithin{equation}{section}
\def\ev{{\mathrm{ev}}}

\def\QQ{{\mathbb{Q}}}

\def\ZZ{{\mathbb{Z}}}
\def\NN{{\mathbb{N}}}
\def\PP{{\mathbb{P}}}
\def\KK{{\mathbb{K}}}
\def\FF{{\mathbb{F}}}

\def\LL{{\mathbb{L}}}
\def\mdeg{{\textrm{mdeg}}}

\def\gcd{{\mathrm{gcd}}}

\def\Res{{\mathrm{Res}}}
\def\Spect{{\mathrm{Spect}}}

\title[On Noether's forms associated to rational functions]{Noether's forms for the study of non-composite rational functions and their spectrum}

\date{\today}

\author[L. Bus\'e]{Laurent Bus\'e}
\address{INRIA, Galaad, 2004 route des Lucioles, B.P. 93\\
Sophia Antipolis\\
06902\\
}
\email{Laurent.Buse@inria.fr}

\author[G. Ch\`eze]{Guillaume Ch\`eze}
\address{ Institut de Mathématiques de Toulouse\\
Université Paul Sabatier Toulouse 3 \\
MIP Bât 1R3,
31 062 TOULOUSE cedex 9, France}
\email{guillaume.cheze@math.ups-tlse.fr}

\author[S. Najib]{Salah Najib}
\address{Max-Planck Institut f\"ur Mathematik. 
Vivatsgasse 7, Bonn 53111. Germany}
\email{najib@mpim-bonn.mpg.de\\ salah.najib@math.univ-lille1.fr}

\keywords{Spectrum, GCD, Noether's theorem,  Ostrowski's theorem, Bertini's theorem, composite rational function}

\begin{document}

	\begin{abstract} In this paper, the spectrum and the decomposability of a multivariate rational function are studied by means of the effective Noether's irreducibility theorem given by Ruppert in \cite{Ru1}. With this approach, some new effective results are obtained. In particular, we show that the reduction modulo $p$ of the
	spectrum of a given integer
	 multivariate rational function $r$ coincides
	with the spectrum of the reduction of $r$ modulo $p$ for $p$ a
	prime integer greater or equal to an explicit bound. This bound is
	given in terms of the degree, the height and the number of variables of $r$. With the same strategy, we also study the decomposability of $r$ modulo $p$. Some similar explicit results are also provided for the case of polynomials with coefficients in $A=\KK[\underline{Z}]$.
	\end{abstract}
	
\maketitle

\section*{Introduction}
Consider a non-constant polynomial $f\in \KK[X_1,\ldots,X_n]$, $n\geq 2$, where $\KK$ is a field. Denoting by $\overline \KK$ the algebraic closure of $\KK$, the \emph{spectrum} of $f$ is the set
$$\sigma(f):= \{\lambda \in \overline {\KK}:
f - \lambda  \hskip 2mm \mbox{is reducible in}\hskip 2mm \overline
{\KK}[X_1, \ldots, X_n]\} \subset \overline{\KK}.$$ 
We recall that a polynomial $f(X_1,\dots,X_n) \in \KK[X_1,\dots,X_n]$ is said to be absolutely irreducible if it is irreducible in $\overline{\KK}[X_1,\dots,X_n]$.

It is customary to say that $f$ is \emph{non-composite} if $f$ cannot be written
in the form $u(h(\underline X))$ with $h(\underline X)\in
\KK[\underline X]$, $u(t)\in \KK[t]$ and $\deg(u)\geq 2$. A famous theorem of Bertini gives that $f$ is non-composite if and only if $\sigma(f)$ is finite; see for instance \cite[Theorem 37]{Sc}.  Furthermore, Stein proved in \cite{St} that the cardinality of $\sigma(f)$ is at most equal to $\deg(f)- 1$; see also \cite{Na2,Lo,Cy,Kal}.

Recently in \cite{BDN}, A. Bodin, P. D\`ebes, and S. Najib have studied  the behavior of the spectrum of a polynomial via a ring morphism. Here we generalize this study to the spectrum of a rational function and we give explicit bounds.

\medskip

Suppose given two non-constant relatively prime polynomials $f$ and $g$ in the polynomial ring $\KK[X_1,\ldots,X_n]$, $n\geq 2$.  The \emph{spectrum} of the rational function $r=f/g \in \KK(X_1,\ldots,X_n)$ is the set
$$\sigma(f,g):= \{(\lambda:\mu) \in \PP^1_{\overline{\KK}}:
	\mu f^\sharp - \lambda g^\sharp   \textrm{ is reducible in }
	\overline {\KK}[X_0,X_1, \ldots, X_n]\} \subset
	\PP^1_{\overline{\KK}}$$
	 with
	$$f^\sharp:=X_0^{\deg(r)}f\left(\frac{X_1}{X_0},\ldots,\frac{X_n}{X_0}\right), \ g^\sharp:=X_0^{\deg(r)}g\left(\frac{X_1}{X_0},\ldots,\frac{X_n}{X_0}\right),$$
where $\deg(r):= \max(\deg(f), \deg(g))$.
That is to say,
\begin{multline*}
\sigma(f,g):= \{(\lambda:\mu) \in \PP^1_{\overline{\KK}} \ : \ 
\mu f - \lambda g  \hskip 2mm \mbox{ is reducible in }\hskip 2mm
\overline {\KK}[X_1, \ldots, X_n]\\
 \textrm{ or } \deg \mu f - \lambda g <\deg r\}.
\end{multline*}
In a more geometric terminology,
$\sigma(f,g)$ counts the number of reducible hypersurfaces in the
pencil of hypersurfaces defined by the equations $\mu f - \lambda
g=0$ with $(\lambda:\mu) \in \PP^1_{\overline{\KK}}$. 

\medskip 

Again, $r$
is said to be \emph{non-composite} if $r$ cannot be written in
the form $u(h(\underline X))$ with $h(\underline X)\in
\KK(\underline X)$ and  $u(t)\in \KK(t)$, $\deg(u)\geq 2$. Actually, $\sigma(f,g)$ is finite if and only if $r$ is non-composite and if and only if the pencil of projective algebraic hypersurfaces $\mu f^\sharp - \lambda g^\sharp =0$, $(\mu:\lambda) \in \PP^1_{\KK}$, has its  general element irreducible (see for instance \cite[Chapitre 2, Th\'eor\`eme 3.4.6]{J79} or \cite[Theorem 2.2]{Bo} for detailed proofs). Notice that the study of $\sigma(f,g)$ is trivial if $d=1$, and $n=1$. Therefore, throughout this paper we will always assume that $d\geq 2$, and $n\geq 2$.

\medskip
The study of the spectrum is related to the computation of  the number of reducible curves in a pencil of algebraic plane curves. This problem has been widely studied. As far as we know, the first related result has been given by Poincar\'e \cite{Poin}.  Poincar\'e's bound was improved by a lot of works, see e.g. \cite{Ru1,Lo,Vi,AHS,Bo}. In \cite{Che_Bus} the authors study the number (counted with multiplicity) of reducible curves in a pencil of algebraic plane curves. The method used relies on an effective Noether's irreducibility theorem given by W. Ruppert in \cite{Ru1}. 

\medskip 

In this article, we follow the same strategy as in \cite{Che_Bus} using in addition basic results of elimination theory. More precisely, in the first section we give some preliminaries which are used throughout this work. In the second section, we show that the spectrum consists of the roots of a particular homogeneous polynomial denoted $\Spect_{f,g}$. If $\varphi$ is a morphism then we get, under some suitable hypothesis, that  $\varphi(\Spect_{f,g})=\Spect_{\varphi(f),\varphi(g)}$. For two special situations, namely $f,g \in \ZZ[\underline{X}]$, and $\varphi$ is the reduction modulo a prime number $p$, or $f,g \in \KK[Z_1,\dots,Z_s][\underline{X}]$, and $\varphi$ sends $Z_i$ to $z_i \in \KK$, we give explicit results  in terms of the degree, the height and the number of variables of $f/g$. 

In the last section of this article we study the behavior of a composite rational function. More precisely we show, under some suitable hypothesis, that ``$f/g$ is composite over its coefficients  field'' if and only if ``$f/g$ is composite over any extension of its coefficients  field''. Thanks to the effective Noether's irreducibility theorem, we then show that if $r$ is a non-composite rational function with integer coefficients and $p$ is a ``big enough'' prime, then $r$  modulo $p$ is also non-composite. An explicit lower bound is given for such an integer $p$. Finally, with the same approach we also study the reduction of a non-composite rational function with coefficients in $\KK[Z_1,\dots,Z_n]$ after the specialization $Z_i=z_i \in \KK$, $i=1,\ldots,n$. We end the paper by showing that after a generic linear change of variables, a non-composite rational function remains non-composite.

\subsection*{Notations}
If $$f(X_1,\dots,X_n)=\sum_{i_1,\dots,i_n} c_{i_1,\dots,i_n} X_1^{i_1}\dots X_n^{i_n}
\in \ZZ[X_1,\dots,X_n]$$ then we set 
$H(f)=\max_{i_1,\dots,i_n} |c_{i_1,\dots,i_n}|$ and 
$\|f\|_1=\sum_{i_1,\dots,i_n} |c_{i_1,\dots,i_n}|$.
The field with $p$ elements $\ZZ/p\ZZ$ is denoted by $\FF_p$.
Given a polynomial $f \in \ZZ[\underline{X}]$ and a prime integer $p \in \ZZ$, we will use the notation $\overline{f}^p$ for the reduction of $f$ modulo $p$, that is to say, the class of $f$ in $\FF_p[\underline{X}]$.
Finally, for any field $\KK$ we denote by $\overline{\KK}$ (one of)
its algebraic closure.

\section{Preliminaries}\label{algtools}

This section is devoted to the statement of some algebraic
properties that are deeply rooted to elimination theory.

\subsection{Noether's reducibility forms} 
We recall some effective results about the Noether's
forms that give a necessary and sufficient condition on the coefficients of a polynomial to be absolutely irreducible. We refer the reader to \cite{Schmidt,Sc,Ka,Ru1} for different types of such forms.

\begin{Thm}\label{Noeth}
Let $\KK$ be a field of characteristic zero, $d \geq 2$, $n \geq 2$ and
$$ f(X_1,\dots,X_n)=\sum_{0\leq e_1+\dots+e_n\leq d} c_{e_1,\dots,e_n} X_1^{e_1}
\dots X_n^{e_n} \in \KK[X_1,\dots,X_n].$$
There exists a finite set of polynomials:
$$\Phi_t(\dots,C_{e_1,\dots,e_n},\dots) \in \ZZ[\dots,C_{e_1,\dots,e_n},\dots]$$
such that
\begin{align*}
\forall t, \Phi_t(\dots,c_{e_1,\dots,e_n},\dots) =0 &\iff  f \textrm{ is reducible in }
\overline{\KK}[X_1,\dots,X_n] \textrm{ or } \deg(f)<d, \\
&\iff F(X_0,\ldots,X_n) \textrm{ is reducible in } \overline{\KK}[X_0,\dots,X_n],
\end{align*}
where $F$ is the homogeneous polynomial $X_0^{\deg(f)}f\left(\frac{X_1}{X_0},\ldots,\frac{X_n}{X_0}\right).$
Furthermore,
$$\deg \Phi_t \leq d^2-1 \textrm{ and }
\| \Phi_t \|_1  \leq  d^{3d^2-3}\Big( \binom{n+d}{n} 2^d \Big)^{d^2-1} .$$
If $\KK$ has positive characteristic $p>d(d-1)$, the coefficients of the $\Phi_t$'s
above are to be
taken modulo $p$.
\end{Thm}

\begin{proof}
In characteristic zero this theorem has been proved by Ruppert \cite{Ru1}. 
Actually Gao showed in \cite{Ga} that this theorem is also true in positive characteristic $p$ providing $p>d(d-1)$. This improvement of original Ruppert's result is contained in \cite[Lemma 2.4]{Ga} where it is explained that the non-vanishing of a certain resultant guaranties the statement of this theorem.
\end{proof}

It should be noticed that the above theorem is true without any hypothesis on the characteristic of the ground field $\KK$, see e.g.~\cite[Theorem 7]{Ka}, but then the degrees and the heights of these forms are with no comparison with the ones given here.

\subsection{GCDs of several polynomials under specialization}

The following theorem is a classical and old result taken from elimination theory. Modern statements and proofs of this result can be found in \cite[\S 2.10]{Jo} and \cite[Corollaire of Th\'eor\`eme 1]{La}.

\begin{Thm}\label{Ma} Let $A$ be a domain and $f_1,\ldots,f_n$ be $n\geq 2$
homogeneous polynomials in $A[U,V]$ of degree $d_1\geq \cdots \geq d_n \geq 1$ respectively.
The polynomials $f_1,\ldots,f_n$ have a common root in the projective line over the algebraic closure of the
fraction field of $A$ if and only if the multiplication map\footnote{the notation $A[U,V]_{d}$, $d\in \NN$, stands for the free $A$-module of homogeneous polynomials of degree $d$.}
$$ \bigoplus_{i=1}^n A[U,V]_{d_1+d_2-d_i-1} \xrightarrow{\varphi} A[U,V]_{d_1+d_2-1} :
(g_1,\ldots,g_n) \mapsto \sum_{i=1}^n g_if_i$$
does not have (maximal) rank $d_1+d_2$.

In particular, given a field $\LL$ and a ring morphism $\rho:A\rightarrow \LL$, $\rho(f_1),\ldots,\rho(f_n)$ have a common root in the projective line over $\overline{\LL}$ if and only if $\rho(\varphi)$ does not have (maximal) rank $d_1+d_2$ (notice that we also denoted by $\rho$ its canonical extensions to some suitable  matrix and polynomial rings).
\end{Thm}

This theorem allows to control the behavior of GCDs of several polynomials with coefficients in a UFD ring under specialization.
Hereafter, 
we will always assume that polynomial GCDs over a field are taken to be monic with respect to a certain monomial order (e.g.~the lexicographical order).

\begin{Cor}\label{UnivGCD} Let $A$ be a UFD, $f_1,\ldots,f_n$ be $n\geq 2$ nonzero homogeneous polynomials in $A[U,V]$ and suppose given a ring morphism $\rho:A \rightarrow \LL$ where $\LL$ is a field. Let $\alpha \in A$ be the leading coefficient of  $\gcd(f_1,\ldots,f_n) \in A[U,V]$.  
There exists a finite collection of algorithmically computable elements $(c_i)_{i \in I}$ in $A$ with the following property:  if $\rho(c_i)\neq 0$ for some $i\in I$ then 
$$ \rho(\gcd(f_1,\ldots,f_n))=\rho(\alpha)\gcd(\rho(f_1),\ldots,\rho(f_n)) \in \LL[U,V].$$
\end{Cor}
\begin{proof}
Set $g_\rho=\gcd(\rho(f_1),\ldots,\rho(f_n)) \in \LL[U,V]$, which is a monic polynomial, and $g=\gcd(f_1,
\ldots,f_n) \in A[U,V]$. 
For all $i=1,\ldots,n$ there exists a polynomial $h_i \in A[U,V]$
such that $f_i=gh_i$. It follows that $\rho(f_i)=\rho(g)\rho(h_i)$ and hence that
$\rho(g)$ divides $g_\rho$. 
Furthermore, $h_1,\ldots,h_n$ have no homogeneous common factor of
positive degree in $A[U,V]$, so by Theorem \ref{Ma} there exists a multiplication map, say $\varphi$, with the property that $\rho(h_1),\ldots,\rho(h_n)$ have no homogeneous common factor in $\LL[U,V]$ if the rank of $\rho(\varphi)$ is maximal. Denoting $(c_i)_{i\in I}$ the collection of maximal minors of a matrix of $\varphi$, the fact that the rank of $\rho(\varphi)$ is not maximal is equivalent to the fact that $\rho(c_i)=0$ for all $i\in I$. Therefore, we deduce that $\rho(g)$ and $g_\rho$ are equal in $\LL[U,V]$ up to an invertible element if $\rho(c_i)\neq 0$ for some $i\in I$. Since $g_\rho$ is monic by convention, the claimed equality is obtained by comparison of the leading coefficients.
\end{proof}

In this paper we will be mainly interested in two particular cases, namely the case where $A=\ZZ$ and $\rho$ is the reduction modulo $p$, and the case where $A=\KK[Z_1,\dots,Z_s]$ and $\rho:A \rightarrow \KK$ is an evaluation morphism. Our next task is to detail Corollary \ref{UnivGCD} in these two particular situations. 

\begin{Prop}\label{Za} Suppose given $n \geq 2$  homogeneous polynomials $f_1,\ldots,f_n \in
\ZZ[U,V]$ of positive degree and set $d=\max_i \deg(f_i)$, $H=\max_i H(f_i)$. 
\begin{itemize}
\item[(i)] If $f_1,\ldots,f_n$ have no (homogeneous) 
common factor of positive degree in $\ZZ[U,V]$, then  $\overline{f_1}^p,\ldots$,
$\overline{f_n}^p$ have no  (homogeneous) common factor of positive degree in $\FF_p[U,V]$
for all prime integer $$p>d^dH^{2d}.$$
\item[(ii)] Denoting by $\alpha \in \ZZ$  the leading coefficient of $\gcd(f_1,\ldots,f_n) \in \ZZ[U,V]$, we have 
$$\overline{\alpha}^p.\gcd(\overline{f_1}^p,\ldots,\overline{f_n}^p) = \overline{\gcd(f_1,\ldots,f_n)}^p
\in \FF_p[U,V]$$
for all prime integer 
$$p > d^d(d+1)^d 2^{2d^2}H^{2d}.$$ 
\end{itemize}
\end{Prop}
\begin{proof} Denote by $d_1\geq d_2\geq \ldots \geq d_n \geq 1$ the degree of $f_1,\ldots,f_n$ respectively. Observe that $d=d_1$. By Theorem \ref{Ma}, the hypothesis implies that the  multiplication map
$$ \oplus_{i=1}^n A[U,V]_{d_1+d_2-d_i-1} \rightarrow A[U,V]_{d_1+d_2-1} : (g_1,\ldots,g_n) \mapsto \sum_{i=1}^n g_if_i$$
has maximal rank $d_1+d_2$. Using Hadamard's inequality \cite[Theorem 16.6]{GG} we obtain that the absolute value of each $(d_1+d_2)$-minor of the matrix of the above multiplication map is bounded above by
$$(\sqrt{dH^2})^{d_1+d_2}=d^{\frac{d_1+d_2}{2}}H^{d_1+d_2} \leq d^dH^{2d}.$$
Therefore, one of these minors remains non-zero modulo $p$ and
Theorem \ref{Ma} implies that
$\overline{f_1}^p,\ldots,\overline{f_n}^p$ do not have a common
root in the projective line over an algebraically closed field extension of $\FF_p$. We
deduce that $\overline{f_1}^p,\ldots,\overline{f_n}^p$ do not have
a homogeneous common factor of positive degree in $\FF_p[U,V]$ and (i) is proved.

To prove (ii) we use Corollary \ref{UnivGCD} and take again the notation of its proof. For all $i=1,\ldots,n$ there exists a homogeneous polynomial $h_i \in \ZZ[U,V]$
such that $f_i=gh_i$ and we know that the claimed equality holds if $\varphi(h_1),\ldots,\varphi(h_n)$ does not have a homogeneous common factor of positive degree in $\FF_p[U,V]$. Since for all $i=1,\ldots,n$ we have $\deg(h_i)\leq d$ and $H(h_i)\leq (d+1)^{\frac{1}{2}}2^dH$ by Mignotte's bounds \cite[Corollary 6.33]{GG} and (i) imply 
that the above condition is satisfied if 
$$p > e^{d\ln(d)} \left[ (d+1)^{\frac{1}{2}} 2^d H \right]^{2d}=e^{d\ln(d)}(d+1)^d 2^{2d^2}H^{2d}.$$
\end{proof}

It should be noticed that a result similar to (i) has been proved in \cite[last paragraph of page 136]{Za} but with a larger bound for the prime integer $p$, namely $e^{2nd^2}H^{2d}$. Also, 
to convince the reader that the bound given in (ii) is not too rough, we mention that in the case $n=2$ it is not difficult to see  that (see for instance \cite[Theorem 6.26]{GG} or \cite[\S 4.4]{Yap})
\begin{equation}\label{gcdn2}
\overline{\gcd(f_1,f_2)}^p=\overline{\alpha}.\gcd(\overline{f_1},\overline{f_2}) \in \FF_p[U,V]
\end{equation}
\emph{if and only if} $p \nmid \Res(h_1,h_2) \in \ZZ$, where $f_i=\gcd(f_1,f_2)h_i$, $i=1,2$ and $\Res(h_1,h_2)$ is the resultant of $h_1$ and $h_2$.
Therefore, it appears necessary to bound $H(h_i)$, $i=1,2$, in terms of $H(f_i)$, $i=1,2$.

\medskip

Now, we turn to the second case of application of Corollary \ref{UnivGCD}. For that purpose, we introduce a new set of indeterminates $\underline{Z}:=Z_1,\ldots,Z_s$.

\begin{Prop}\label{KZ} Let $f_1,\ldots,f_n$ be $n\geq 2$ polynomials in $\KK[\underline{Z}][U,V]$ that are homogeneous with respect to the variables $U,V$ of degree $d$ with coefficients in $\KK[\underline{Z}]$. Also, suppose given a ring morphism $\rho: \KK[\underline{Z}] \rightarrow \KK$ and assume that all coefficients of $f_1,\ldots,f_n$ are polynomials in $\KK[\underline{Z}]$ of degree $\leq k$.
\begin{itemize}
\item[(i)] If $f_1,\ldots,f_n$ have no (homogeneous)
common factor of positive degree in $\KK[\underline{Z}][U,V]$, then  
there exists a finite collection $(p_i)_{i\in I}$ of nonzero elements in $\KK[\underline{Z}]$ of degree $\leq 2dk$  such that $\rho(f_1),\ldots$,
$\rho(f_n)$ have no (homogeneous) common factor of positive degree in $\KK[U,V]$ if $\rho(p_i)$, $i\in I$, are not all zero. 
\item[(ii)] Denoting by $\alpha \in \KK[\underline{Z}]$ the leading coefficient of $\gcd(f_1,\ldots,f_n) \in \KK[\underline{Z}][U,V]$ as a homogeneous polynomial in the variables $U,V$, there exists a finite collection $(q_i)_{i\in I}$ of nonzero elements in $\KK[\underline{Z}]$
of degree $\leq 2dk$ such that $$\rho(\alpha).\gcd(\rho(f_1),\ldots,\rho(f_n)) = \rho(\gcd(f_1,\ldots,f_n))
\in \KK[X]$$ if $\rho(q_i)$, $i\in I$, are not all zero.
\end{itemize}
\end{Prop}
\begin{proof} Completely similar to the proof of Proposition \ref{Za}. 
\end{proof}


\section{Study of the spectrum of a rational function}

Let $A$ be a UFD, $\KK$ be its fraction field and suppose given a rational function 
$r=f/g \in \KK(X_1,\ldots,X_n)$ with $f,g \in A[X_1,\dots,X_n]$ and $\gcd(f,g)=1$. 
Set $d:=\deg(r)=\max(\deg(f),\deg(g))$. We recall that, by definition, the \emph{spectrum} of $r$ is the set
	$$\sigma(f,g):= \{(\lambda:\mu) \in \PP^1_{\overline{\KK}}:
	\mu f^\sharp - \lambda g^\sharp   \textrm{ is reducible in }
	\overline {\KK}[X_0,X_1, \ldots, X_n]\} \subset
	\PP^1_{\overline{\KK}}$$
	 where
	$$f^\sharp:=X_0^{\deg(r)}f\left(\frac{X_1}{X_0},\ldots,\frac{X_n}{X_0}\right), \ g^\sharp:=X_0^{\deg(r)}g\left(\frac{X_1}{X_0},\ldots,\frac{X_n}{X_0}\right).$$

Assume that $\sigma(f,g)$ is finite and denote by 	$\Phi_t(U,V)$ the Noether's reducibility forms associated to the polynomial 
	$$Vf^\sharp(X_0,\ldots,X_n)-Ug^\sharp(X_0,\ldots,X_n) \in A[U,V][X_0,X_1,\ldots,X_n].$$
These forms are all homogeneous polynomials in $A[U,V]$ by construction. We will denote their GCD by $\Spect_{f,g}(U,V)\in A[U,V]$. By Theorem \ref{Noeth}, 
for all $(\lambda:\mu) \in \PP^1_{\overline{K}}$ we have
	$$\Spect_{f,g}(\lambda:\mu)=0 \iff (\lambda:\mu) \in \sigma(f,g).$$

As an immediate consequence of Corollary \ref{UnivGCD}, we have the following property. Given a morphism $\rho: A \rightarrow \LL$, where $\LL$ is a field, there exists a non-zero element $c$ of $A$ such that if $\rho(c) \neq 0$ then
$$\rho(\Spect_{f,g})=\Spect_{\rho(f),\rho(g)}$$
up to multiplication by a non-zero element in $\LL$.

In what follows, we will investigate this property in two particular cases of interest: the case where $A=\ZZ$ and $\rho$ is the reduction modulo $p$, and the case where $A=\KK[Z_1,\dots,Z_s]$ and $\rho:A \rightarrow \KK$ is an evaluation morphism.

\subsection{Spectrum and reduction modulo $p$}

\begin{Thm}\label{modp} Suppose given $f,g$ two multivariate polynomials in $\ZZ[X_1,\dots,X_n]$ such that $\gcd(f,g)=1$. Set $d=\deg(f/g)=\max(\deg(f),\deg(g))$. For all prime integer $p>\mathcal{B}$ with
$$\mathcal{B}=2^{2(d^2-1)^2} d^{2d^2-2} (d^2-1)^{d^2-1}  \mathcal{H}^{2d}$$ where
$$\mathcal{H}=d^{3d^2-3}\Big( \binom{n+d}{n} 2^d \Big)^{d^2-1} \binom{d^2-1}{\lfloor(d^2-1)/2\rfloor}\max\big(H(f),H(g) \big)^{d^2-1}$$
we have 
$$Spect_{\overline{f}^p,\overline{g}^p}= \kappa . \overline{Spect_{f,g}}^p$$
in the polynomial ring $\FF_p[U,V]$ with $0\neq \kappa \in \FF_p$.
\end{Thm}

\begin{proof} From the definition of $Spect_{f,g}(U,V)$, this is a consequence of Theorem \ref{Noeth}. Indeed, straightforward computations show that

$$H(\Phi_t( Vf^\sharp- Ug^\sharp ))\leq \| \Phi_t \|_1 H\big( (Vf^\sharp- Ug^\sharp )^{d^2-1} \big)$$
and 
$$H \big(  (Vf^\sharp- Ug^\sharp )^{d^2-1} \big) \leq \binom{d^2-1}{\lfloor (d^2-1)/2  \rfloor }\max\big(H(f),H(g) \big)^{d^2-1}.$$
It follows
\begin{multline*}
H(\Phi_t( Vf^\sharp- Ug^\sharp ))\leq \\	
d^{3d^2-3}\Big( \binom{n+d}{n} 2^d \Big)^{d^2-1} \binom{d^2-1}{\lfloor(d^2-1)/2\rfloor}\max\big(H(f),H(g) \big)^{d^2-1}.
\end{multline*}
Now, applying Proposition \ref{Za} with degree $d^2-1$ and height $H_1$ we obtain that 
$$\overline{Spect_{f,g}(U,V)}^p  =  \overline{\gcd \big( \Phi_t(Vf^\sharp-Ug^\sharp) \big)}^p
= \kappa.\gcd\big( \overline{\Phi_t( Vf^\sharp -T g^\sharp)}^p  \big) $$
if $p>\mathcal{B}$, where $0\neq \kappa \in \FF_p$, and therefore  that 
$$ \overline{Spect_{f,g}(U,V)}^p  = 
\kappa . Spect_{\overline{f}^p,\overline{g}^p}(U,V)
$$ by Theorem \ref{Noeth}.
\end{proof}

As a consequence of this theorem we obtain an analog of Ostrowski's result. The classical Ostrowski's theorem asserts that if a polynomial $f$ is absolutely irreducible then $\overline{f}^p$ is absolutely irreducible providing $p$ is big enough. In our context we get the 

\begin{Cor}
Under the notations of Theorem \ref{modp}, if $\sigma(f,g)= \emptyset$ then $\sigma(\overline{f}^p,\overline{g}^p)= \emptyset$ for all prime integer $p$ such that $p>\mathcal{B}$.
\end{Cor}

Before moving on, we mention that our
strategy can be used similarly to deal with the case of polynomials in $A[X_1,\dots,X_n]$ and reduction modulo a prime ideal in $A$.

\subsection{Spectrum of a rational function with  coefficients in $\KK[\underline{Z }]$}

\begin{Thm}\label{ThKZ}
Let $f,g$ be two polynomials in  $\KK[Z_1,\dots,Z_s][X_1,\dots,X_n]$ such that $\deg_{\underline{Z}}(f)\leq k$, $\deg_{\underline{Z}}(g)\leq k$,  $\deg_{\underline{X}}(f)\leq d$ and $\deg_{\underline{X}}(g)\leq d$. Given $\underline{z}:=(z_1,\ldots,z_s) \in \KK^s$, denote by $\ev_{\underline{z}}$ the ring morphism $\KK[Z_1,\dots,Z_s] \rightarrow \KK$ that sends $Z_i$ to $z_i$ for all $i=1,\ldots,s$. 
There exists a finite collection of nonzero polynomials in $\KK[\underline{Z}]$, say $(q_i)_{i\in I}$, of degree smaller than $2(d^2-1)^2k$
with the property that: if $\ev_{\underline{z}}(q_i) \in \KK$ are not all zero then
$$ev_{\underline{z}}(\Spect_{f,g})=\kappa.\Spect_{ev_{\underline{z}(f)},ev_{\underline{z}(g)}}$$
where $0\neq \kappa \in \KK$.
\end{Thm}

\begin{proof}
We consider again  Noether's forms $(\Phi_t(Vf^\sharp-Ug^\sharp))_{t\in T}$. By construction we have $\deg_{\underline{Z}}( \Phi_t(Vf^\sharp-Ug^\sharp)) \leq (d^2-1)k$ and $\deg_{U,V}( \Phi_t(Vf^\sharp-Ug^\sharp)) \leq d^2-1.$ Therefore, Proposition \ref{KZ} shows the existence of a finite collection of polynomials $(q_j)_{j\in J}$ in $\KK[\underline{Z}]$ of degree smaller than $2(d^2-1)^2k$ with the property that $$ev_{\underline{z}}(\Spect_{f,g})=\kappa.\Spect_{ev_{\underline{z}(f)},ev_{\underline{z}(g)}}$$
where $0\neq \kappa \in \KK$, if $\ev_{\underline{z}}(q_j)\neq 0$ for some $j\in J$.
\end{proof}

This result has the following probabilistic corollary that follows from the well known Zippel-Schwartz's Lemma that we recall.
\begin{Lem}[Zippel-Schwartz]\label{lem_zippel} 
Let $P\in A[X_1,\ldots,X_n]$ be a polynomial of total degree $d,$
where $A$ is an integral domain. Let $S$ be a finite subset of $A$.
For a uniform random choice of $x_i$ in $S$ we have
\[\mathcal{P} \Big( \{P(\underline x)=0  \mid x_i \in S\} \Big) \leq d/|S|,\]
where $|S|$ denotes the cardinal of $S$ and $\mathbf{\mathcal{P}}$ the probability.
\end{Lem}

\begin{Cor} With the notations of Theorem \ref{ThKZ}, let $S$ be a finite subset of $\KK$ with $|S|$ elements. If $\sigma(f,g)
=\emptyset$ then for a uniform random choice of the $z_i$'s in $S$ we have 
$$\sigma(ev_{\underline{z}(f)},ev_{\underline{z}(g)})=\emptyset$$
with probability at least  $1- \dfrac{2(d^2-1)^3k^2}{|S|}$.
\end{Cor}
\begin{proof}
As $\sigma(f,g)=\emptyset$, $\Spect_{f,g}=:c(\underline{Z})$ is a non-zero polynomial in $\KK[\underline{Z}]$ of degree less than $k(d^2-1)$. Therefore, if $q_ic(\underline{z})\neq 0$ for some $i \in I$, where $(q_i)_{i \in I}$ is the collection of polynomials in Theorem \ref{ThKZ}, then $\Spect_{ev_{\underline{z}(f)},ev_{\underline{z}(g)}} \in \KK$.
\end{proof}

\section{Indecomposability of rational functions}
In the previous section we have studied the spectrum of a rational function. It turns out that the spectrum of $r(X_1,\dots,X_n) \in \KK(X_1,\dots,X_n)$ is deeply related to the indecomposability of $r$ over $\overline{\KK}$. After recalling this link, we will study indecomposability of rational functions.

\begin{Thm}\label{frac} Let $\KK$ be a field of characteristic $p \geq 0$.
Let $r= f/g \in \KK(X_1,\dots,X_n)$ be  a non-constant reduced rational function. The following are equivalent:
\begin{itemize}
\item[(i)] $r$ is composite over $\overline \KK$,
\item[(ii)]  $f - \lambda g$ is reducible in $\overline {\KK}[X_1,\dots,X_n]$ for all $\lambda \in \overline{\KK}$ such that $\deg(f-\lambda g)= \deg(r),$
\item[(iii)] the polynomial $f(X_1,\dots,X_n)- Tg(X_1,\dots,X_n)$ is
reducible in the polynomial ring $\overline {\KK(T)}[X_1,\dots,X_n]$.
\end{itemize}
\end{Thm}

\begin{Rem}
We recall that this kind of result is already known
for the polynomial case (see for instance \cite[Lemma 7]{CN}).
\end{Rem} 

\begin{proof}[Proof of Theorem  \ref{frac}]
(i)$\iff$ (ii): see \cite[Theorem 2.2]{Bo}.

\medskip

\noindent (ii) $\Rightarrow$ (iii): statement (ii) means that for all $\lambda \in \overline{\KK}$ such that $\deg(f -\lambda g)=\deg r$, we have $\Phi_t(f-\lambda g)=0$ for all $t$. Then $\Phi_t(f-Tg)$ has an infinite number of roots in $\overline{\KK}$ and thus $\Phi_t(f-Tg)=0$ for all $t$. This gives $f-Tg$ is reducible in $\overline{\KK(T)}[X_1,\dots,X_n]$.

\medskip

\noindent (iii) $\Rightarrow $ (ii): statement  
(iii) means that $\Phi_t(f-Tg)=0 \in \KK[T]$ for all $t$. Hence, if $\lambda \in \overline{\KK}$ is such that $\deg(f -\lambda g)=\deg r$, we can conclude thanks to Theorem \ref{Noeth}  that $f-\lambda g$ is reducible in $\overline{\KK}[X_1, \dots,X_n]$.
\end{proof}

The following theorem shows that under some hypothesis, $r$ is composite over $\KK$ if and only if $r$ is composite over $\overline{\KK}$. Therefore, we will sometimes say hereafter that $r$ is composite instead of $r$ is composite over its coefficients field.

\begin{Thm}\label{comp_k_kbar} Let $\KK$ be a perfect  field of characteristic $p=0$ or $p \geq d^2$ and 
let $r=f/g \in \KK(X_1,\dots,X_n)$, $n\geq 2$, be a non-constant reduced rational function of degree $d$. Then, 
$r$ is composite over $\KK$ if and only if $r$ is composite over $\overline{\KK}$.
\end{Thm}

\begin{proof} Obviously, 
if $r$ is composite over $\KK$ then $r$ is composite over any extension of $\KK$ and thus over $\overline{\KK}$. So, 
suppose that $r=u(h)$ with $\deg(u)\geq 2$, $u\in \overline{\KK}(T)$ and $h\in \overline{\KK}(X_1,\dots,X_n)$. We set $u=u_1/u_2$ where $u_1,u_2 \in \overline{\KK}[T]$ and $h=h_1/h_2$ is reduced and non-composite with $h_1,h_2 \in \overline{\KK}[X_1,\dots,X_n]$. We are going to show that there exist $U \in \KK(T)$ and $H \in \KK(X_1,\dots,X_n)$ such that $r=U(H)$.

The notation $\mdeg(f)$ denotes the multi-degree of $f$ associated to a given monomial order $\prec$ and $lc(f)$ denotes the leading coefficient of $f$ associated to $\prec$.

\medskip

\noindent \textsf{First step:} we claim that one can suppose that $lc(f)=1$, $lc(g)=1$ and $\mdeg(f)\succ\mdeg(g)$. Indeed, to satisfy the first condition, we just have to take $f/lc(f)$ and $g/lc(g)$. Then, for the second conditon, if $\mdeg(f)\prec \mdeg(g)$ we take $g/f$, and if $\mdeg(f)=\mdeg(g)$ we set $F=f$, $G=f-g$ and take $F/G$. Indeed, $f/g$ is composite over $\KK$ if and only if $F/G$ is composite over $\KK$.

\medskip

\noindent\textsf{Second step:} we claim that one can suppose that $lc(h_1)=1$, $lc(h_2)=1$ and $\mdeg(h_1)\succ \mdeg(h_2)$. This actually follows from the same trick that we use in the first step. We just have to remark that if $r=u(h_1/h_2)$ then $r=v(h_2/h_1)$ with $v(T)=u(1/T)$.

\medskip

\noindent\textsf{Third step:} one can suppose that $h_1(0,\ldots,0)=0$ and  $h_2(0,\ldots,0) \neq 0$. Indeed, 
if $h_2(0,\ldots,0)\neq 0$ then we can consider 
$$H_1=h_1-\left( \frac{h_1(0,\ldots,0)}{h_2(0,\ldots,0)} \right) h_2$$ and  $H_2=h_2$. Then we can write $r=v(H_1/H_2)$ with $H_1$ and $H_2$ satisfying the above conditions and the ones of the second step.

Now, if $h_2(0,\ldots,0)=0$ then a change of coordinates $r(X_1-a_1,\ldots,X_n-a_n)$, with $(a_1,\ldots,a_n) \in \KK^n$ such that $h_2(a_1,\ldots,a_n)\neq 0$, gives the desired result. So we just have to show that there exists such an element $(a_1,\ldots,a_n)\in \KK^n$ if $p\geq d^2$ (it is clear if $p=0$). To do this, we observe that $\deg h_2 \leq d <p$ and proceed by contradiction. 

Assume that
\begin{multline*}
	h_2(X_1,\ldots,X_n)=\\
	c_0(X_1,\ldots,X_{n-1})+c_1(X_1,\ldots,X_{n-1})X_n+\cdots+c_d(X_1,\ldots,X_{n-1})X_n^d	
\end{multline*}
with $c_i(X_1,\ldots,X_{n-1}) \in \overline{\KK}[X_1,\ldots,X_{n-1}]$, 
is such that 
$$\forall (x_1,\ldots,x_n)\in \KK^n,\, h_2(x_1,\ldots,x_n)=0.$$ 
Then, given $(x_1,\ldots,x_{n-1}) \in \KK^{n-1}$ we have that $h_2(x_1,\ldots,x_{n-1},X_n) \in \overline{\KK}[X_n]$ has degree $\leq d$ and  at least $p$ distinct roots in $\KK$. It follows that $h_2(x_1,\ldots,x_{n-1},X_n)$ is the null polynomial and hence that $$\forall  i=0,\ldots,d,\, \forall (x_1,\ldots,x_{n-1}) \in \KK^{n-1},\, c_i(x_1,\ldots,x_{n-1})=0.$$
Now, since $c_i(X_1,\ldots,X_{n-1})$ has also degree $\leq d$, we can continue this process in the same way to end with the conclusion that $h_2=0$ in $\overline{\KK}[X_1,\ldots,X_{n}]$.

\medskip

\noindent\textsf{Fourth step:} we claim that $h_2(X_1,\ldots,X_n) \in \KK[X_1,\ldots,X_n]$. To show this, we are going to prove that if $r$, $h_1$ and $h_2$ satisfies the hypothesis of the previous steps then $h_2 \in \KK[X_1,\ldots,X_n]$.

Let $\lambda \in \KK$ such that $\deg(f-\lambda g)=\deg r= d$. Since $r$ is composite over $\overline{\KK}$ then  $f-\lambda g$ is reducible over $\overline{\KK}$ by Theorem \ref{frac}, and we have:
$$f+\lambda g=\alpha \prod_{i=1}^m(h_1+\lambda_i h_2),$$
where $\lambda_i$ are the roots of $u_1+\lambda u_2=\alpha\prod_{i=1}^m(T-\lambda_i)$, see the proof of \cite[Corollary 2.4]{Bo}. 
Thanks to step 1 and 2, $lc(f+\lambda g)=1$ and $lc(h_1+\lambda_i h_2)=1$ so that $\alpha=1$. 
As $h_1/h_2$ is non-composite we can find  $\lambda \in \KK$ such that for all $i$, $h_1+\lambda_i h_2$ is irreducible over $\overline{\KK}$ because $|\sigma(h_1,h_2)|\leq d^2-1$ and $p=0$ or $p>d^2$ (the bound $|\sigma(h_1,h_2)|\leq d^2-1$ is proved for any field in the bivariate case in \cite{Lo} and its extension to the multivariate case is easily obtained using Bertini's Theorem; see for instance \cite[Proof of Theorem 13]{Che_Bus} or \cite{Bo}.)

Now let $\tau \in Galois(\LL/\KK)$, where $\LL$ is the field generated by all the coefficients of $u_1$, $u_2$, $h_1$, $h_2$, we have:
$$f-\lambda g= \tau(f-\lambda g)=\prod_{i=1}^m\big(\tau(h_1)+\tau(\lambda_i)\tau(h_2)\big).$$
As $lc\big(\tau(h_1)+\tau(\lambda_i)\tau(h_2)\big)=1$ and $\tau(h_1)+\tau(\lambda_i)\tau(h_2)$ is also irreducible over $\overline{\KK}$, we can write:
\begin{eqnarray*}
h_1+\lambda_1h_2&=&\tau(h_1)+\tau(\lambda_{i_1}) \tau(h_2), \hspace{1cm} (\star)\\
h_1+\lambda_2h_2&=&\tau(h_1)+\tau(\lambda_{i_2}) \tau(h_2).
\end{eqnarray*}
Thus, $(\lambda_1-\lambda_2)h_2=(\tau(\lambda_{i_1})-\tau(\lambda_{i_2}))\tau(h_2)$. As $lc(h_2)=1$, we deduce that $(\lambda_1-\lambda_2)=\tau(\lambda_{i_1})-\tau(\lambda_{i_2})$ and then $h_2=\tau(h_2)$. This implies that $h_2 \in \KK[X_1,\ldots,X_n]$ because $\KK$ is a perfect field.

\medskip

\noindent\textsf{Last step:} we claim that $h_1 \in \KK[X_1,\ldots,X_n]$. Indeed, 
$(\star)$ and  the hypothesis  $h_1(0,\ldots,0)=0$ (see step 3) implies that 
$\lambda_1h_2(0,\ldots,0)=\tau(\lambda_{i_1})h_2(0,\ldots,0)$. As $h_2(0,\ldots,0)\neq 0$  (by step 3 again), we get $\lambda_1=\tau(\lambda_{i_1})$. Then $(\star)$ means that $h_1=\tau(h_1)$ and this concludes the proof because $\KK$ is a perfect field.
\end{proof}

\begin{Rem} First, notice that the above result remains obviously true when we take any extension of $\KK$ instead of $\overline \KK$. Also,
observe that this theorem is false for univariate\footnote{Recall that a non-constant univariate rational function $r(X)\in \KK(X)$ is called composite over a field $\KK$ if $r(X)=u(h(X)),$ where $u,\  h \in \KK(X)$ such that $\deg(u)\geq 2$ and $\deg(h)\geq 2.$} ($n=1$) rational function; see \cite[Example 5]{GuSe06}. Finally, mention that if $p\leq d^2$ and the field $\KK$ is not perfect then the theorem is also false. Indeed, in 
\cite[page 27]{Ayad} one can find the following counterexample: $f(X,Y)=X^p+bY^p=(X+\beta Y)^p$, with $b \in \KK \setminus \KK^p$ and $\beta^p=b$, is composite in $\KK(\beta)$ (which is clear) but non-composite in $\KK$ (which is proved in loc.~cit.).
\end{Rem}

Theorem \ref{frac} and Theorem \ref{comp_k_kbar} yield the

\begin{Cor}\label{cor_comp_et_Spect}
\begin{eqnarray*}
r=f/g \textrm{ is non-composite } &\iff& \Spect_{f,g}(T) \neq 0 \textrm{ in } \KK[T]\\
&\iff & \sigma(f,g) \textrm{ is finite.}
\end{eqnarray*}
\end{Cor}

Corollary \ref{cor_comp_et_Spect} clearly implies several results about the indecomposability of $r$. For instance, if $r=f/g$ is a non-composite rational function where $f,g \in \ZZ[X_1,\dots,X_n]$, and $p$ is a prime integer bigger than $H(\Spect_{f,g})$ and the bound $\mathcal{B}$ of Theorem \ref{modp}, then $\overline{r}^p$ is non-composite. Indeed, $\overline{\Spect_{f,g}}^p=\Spect_{\overline{f}^p,\overline{g}^p}\neq 0$ in $\FF_p[T]$, for all $p >\mathcal{B}$.

With this strategy we could deduce several similar results but the bounds obtained in this way can be improved. Indeed, when we use the polynomial $\Spect$ we have to study the gcd of the $\Phi_t(f-Tg)$'s. But if $r$ is supposed to be non-composite then there exits an index $t_0$ such that $\Phi_{t_0}(f-Tg) \neq 0$. In this case it is enough to study the behaviour of one polynomial instead of the gcd of several polynomials. Thus, in what follows we are going to study the indecomposability of a rational function using Noether's forms.

\subsection{Reduction modulo p}

\begin{Thm}
Let $r=f/g \in \ZZ(\underline{X})$ be a non-constant reduced and non-composite  rational function. 
If $$p>  \mathcal{H}=d^{3d^2-3}\Big( \binom{n+d}{n} 2^d \Big)^{d^2-1} \binom{d^2-1}{\lfloor(d^2-1)/2\rfloor}\max\big(H(f),H(g) \big)^{d^2-1},$$
then $\overline{r}^p$ is non-composite  and $\overline{f}^p,\ \overline{g}^p$ are coprime.
\end{Thm}

\begin{proof}
Thanks to Theorem \ref{frac}, we have that $f-Tg$ is irreducible in $\overline{\QQ(T)}[X_1,\dots,X_n]$. Therefore, there exists $t_0$ such that $\Phi_{t_0}(f-Tg)\neq 0$ in $\ZZ[T]$. Now, if $p>\mathcal{H}$ then $\overline{\Phi_{t_0}}^p(\overline{f}^p-T\overline{g}^p) \neq 0$ (see the proof of Theorem \ref{modp}). This means that $\overline{f}^p -T \overline{g}^p$ is irreducible in $\overline{\FF_p(T)}[X_1,\dots,X_n]$ and hence $\overline{r}^p$ is non-composite by Theorem \ref{frac}. Of course, $\overline{f}^p$ and $\overline{g}^p$ are coprime because otherwise $\overline{f}^p-T\overline{g}^p$ cannot be irreducible.
\end{proof}

\subsection{Indecomposability of rational function with coefficients in $\KK[\underline{Z}]$}

\begin{Thm}\label{indecomp_zip}
Let $d$ and $k$ be positive integers, $\KK$ be a perfect field of characteristic 0 or $p\geq d^2$,  
$r=f/g \in \KK[\underline{Z}](\underline{X})$ be a non-constant reduced  rational function with  $0<\deg_{\underline X}(r)\leq d$,  $0<\deg_{\underline Z}(r)\leq k$ and let $S$ be a finite subset of $\KK$.

If $r$ is non-composite over $\KK(\underline{Z})$ then 
for a uniform random choice of $z_i$ in $S$ we have
\[\mathcal{P} \Big( \{ r(z_1,\dots,z_s,\underline{X}) \textrm{ is non-composite over } \KK   \mid z_i \in S\} \Big) \geq 1 -k(d^2-1)/|S|,\]
where $|S|$ denotes the cardinal of $S$ and $\mathcal{P}$ the probability.
\end{Thm}

\begin{proof} 
Assume that $r$ is non-composite over $\KK(\underline{Z})$. Then, by Theorem \ref{frac}, we have that $f-Tg$ is irreducible in $\overline{\KK(\underline{Z})}[\underline{X}]$. Thus there exists $t_0$ such that $\Phi_{t_0}(f-Tg)\neq 0$ in $\KK[\underline{Z}][T]$. We can write  $\Phi_{t_0}(f-Tg)=\sum_{i=0}^D a_iT^i$ with $a_i \in \KK[\underline{Z}]$ and $a_D \neq  0 \in \KK[\underline{Z}]$. Therefore, for all $\underline{z} \in \KK^s$ such that $a_D(\underline{z}) \neq 0$ we have $r(\underline{z},\underline{X})$ is non-composite. Furthermore Theorem \ref{Noeth} gives $\deg_Z a_i \leq k(d^2-1)$. Then, Lemma \ref{lem_zippel} applied to $a_D(\underline{Z})$ gives the desired result.
\end{proof} 

\begin{Rem} 
Theorem \ref{indecomp_zip} is false with the hypothesis ``$r$ is non-composite over $\KK$'' instead of ``$r$ is non-composite over $\KK(\underline{Z})$''. Indeed, take $n=2$ and $s=1$ and consider the polynomial $f(X,Y,Z)= (XY)^{2}+ Z.$ This polynomial is non-composite over $\KK$ (because $\deg_{Z}(f)=1$) but $f(X,Y,z)= (XY)^{2} + z$ is composite over $\KK$ for all value $z\in \KK.$
\end{Rem}

\subsection{A reduction from $n$ to two variables}

We give the following Bertini like result.

\begin{Thm}
Let $\KK$ be a  perfect field of characteristic 0 or $p \geq d^2$, $S$ be a finite subset of $\KK$ and let
$r=f/g \in \KK(X_1,\dots,X_n)$ be a reduced non-composite  rational function.

For a uniform random choices of the
$u_i$'s, $v_i$'s and $w_i$'s in $S$,  the rational function
\[\tilde{r}(X,Y)=r(u_1X+v_1Y+w_1, \ldots,u_nX+v_nY+w_n)\in \KK[X,Y].\]
is non-composite 
with probability at least \mbox{$1 -\big(3d(d-1)+1\big)/|S|$} where $d$ is the degree of $r$.
\end{Thm}

\begin{proof}
As we did before, we study $f-Tg$. This polynomial is irreducible over $\overline{\KK(T)}$ by Theorem \ref{frac}. Then we apply the effective Bertini's Theorem given in \cite[Corollary 8]{Le} to this polynomial. We obtain that 
$\tilde{f}(X,Y)-T\tilde{g}(X,Y)$ is irreducible in $\overline{\KK(T)}[X,Y]$ with a probability at least $1 -\big(3d(d-1)+1\big)/|S|$. Then by Theorem \ref{frac} yields the desired result about $\tilde{r}$.
\end{proof}

\section{Acknowledgment}
During the preparation of this paper, the third author was supported by Abdus Salam center (ICTP, Trieste) and after by Max-Planck Institut (Bonn). He wish to thank Pierre D\`ebes for encouragements and comments, as well as Enrico Bombieri and Umberto Zannier for their discussions in the beginning of this work.


\end{document}